\documentclass[12pt,twoside,a4paper]{article}

\usepackage{a4,enumerate}
\usepackage[noadjust]{cite}
\usepackage{amssymb,amsmath,amsthm}
\usepackage{comment}
\usepackage{calc}
\usepackage{xcolor}

\newcommand\lhead{H. Vogt, J. Voigt}
\newcommand\rhead{Bands in $L_p$-spaces}
\swapnumbers
\numberwithin{equation}{section}

\newtheorem{theorem}{Theorem}[section]
\newtheorem{corollary}[theorem]{Corollary}
\newtheorem{proposition}[theorem]{Proposition}

\theoremstyle{definition}

\newtheorem{remark}[theorem]{Remark}

 \mathchardef\ordinarycolon\mathcode`\:
  \mathcode`\:=\string"8000
  \begingroup \catcode`\:=\active
    \gdef:{\mathrel{\mathop\ordinarycolon}}
  \endgroup

\newcommand\rlim{
\mathchoice{\vcenter{\hbox{${\scriptstyle{+}}$}}}
{\vcenter{\hbox{$\scriptstyle{+}$}}}
{\vcenter{\hbox{$\scriptscriptstyle{+}$}}}
{\vcenter{\hbox{$\scriptscriptstyle{+}$}}}}

\newcommand\ran{\operatorname{\rm ran}}
\newcommand\indic{\operatorname{\bf 1}\nolimits}

\newcommand{\lin}{\operatorname{lin}}

\renewcommand\phi{\varphi}

\renewcommand\epsilon{\varepsilon}

\newcommand{\R}{\mathbb{R}\nonscript\hskip.03em}
\newcommand{\N}{\mathbb{N}\nonscript\hskip.03em}

\newcommand\cA{\mathcal A}
\newcommand\cL{\mathcal L}
\newcommand\cE{\mathcal E}
\newcommand\lperp{\perp_\loc}

\newcommand\ndash{\rule[.58ex]{\widthof{--}}{0.065ex}} 

\newcommand\fin{{\rm fin}}
\newcommand\sfin{{\sigma\text{-fin}}}
\newcommand\loc{{\rm loc}}
\newcommand\rma{{\rm (a) }}
\newcommand\rmb{{\rm (b) }}

\newcommand\eul{{\rm e}}
\newcommand\esssup{\operatorname{\rm ess\,sup}}
\newcommand\clg{\mkern-2mu}   

\makeatletter
\let\qedhere@ams\qedhere
\def\qedhere{\@ifnextchar[{\@qedhere}{\qedhere@ams}}
\def\@qedhere[#1]{\tag*{\raisebox{-#1ex}{\qedhere@ams}}}
\makeatletter
\def\env@cases{%
  \let\@ifnextchar\new@ifnextchar
  \left\lbrace
  \def\arraystretch{1.1}%
  \array{@{\,}l@{\quad}l@{}}%
}
\renewcommand\section{\@startsection {section}{1}{\z@}%
                                     {-3.25ex \@plus -1ex \@minus -.2ex}%
                                     {1.5ex \@plus.2ex}%
                                     {\normalfont\large\bfseries}}
\makeatother

\newcommand\restrict{\vphantom f\mskip1mu\vrule\mskip2mu}
\newcommand\set[2]{\bigl\{#1{;}\;#2\bigr\}}
\newcommand\sset[2]{\{#1{;}\;#2\}}
\newcommand\setcol{;\;}

\renewcommand\le{\leqslant}

\renewcommand\ge{\geqslant}

\newcommand{\from}{{:}\;}
\renewcommand{\from}{\colon}

\newcommand\sse{\subseteq}

\newcommand\di{\mathclose{}\,\mathrm{d}}

\newcommand\slim{\mathop{\rm s\kern.08em\mbox{\rm -}lim}} 

\newcommand\abstracttext{\noindent
For a general measure space $(\Omega,\mu)$, it is shown that for every band $M$ 
in $L_p(\mu)$ there exists a decomposition 
$\mu=\mu'+\mu''$ such that $M=L_p(\mu')=\sset{f\in L_p(\mu)}{f=0\ 
\,\mu''\text{-a.e.}}$. The theory is illustrated by an example, with an 
application to absorption semigroups.

\vspace{8pt}

\noindent
MSC 2010: 46B42, 28A05, 47D06
\vspace{2pt}

\noindent
Keywords: band in $L_p$, band projection, decomposition of measures, 
localisable measure space, absorption semigroup
}
\begin{document}
\title{Bands in $L_p$-spaces}

\author{Hendrik Vogt and J\"urgen Voigt}

\date{}

\maketitle

\begin{abstract}
\abstracttext
\end{abstract}

\section*{Introduction}
\label{intro}

Let $(\Omega,\cA,\mu)$ be a measure space (in particular, $\cA$ is assumed to 
be a $\sigma$-algebra), and let $1\le p<\infty$. 
If $\Omega_0$ is a locally measurable subset of $\Omega$, and $M$ is the set of 
all functions in $L_p(\mu)$ vanishing a.e.\ outside $\Omega_0$, then $M$ is a 
band in $L_p(\mu)$ (see the beginning of 
Section~\ref{proj-bands} for the definition). This description holds for all 
bands if 
the measure $\mu$ is $\sigma$-finite (or, more generally, weakly localisable). 
Looking at the case of general measures one 
realises that such a property is true `locally', but the local sets cannot be 
composed to a locally measurable set, in general. Our main result, 
Theorem~\ref{mainthm}, is that this composition can be achieved on the level of 
measures.

The issue sketched above came up in connection with absorption semigroups; 
see~\cite{man-vog-voi-16}. More specifically, for an absorption semigroup $T_V$ 
on $L_p(\mu)$ with a `very singular' absorption rate $V$ it is known that 
$P:=\slim_{t\to0\rlim}T_V(t)$ is a band projection, and we demonstrate 
by an 
example that also in this specific context the band $\ran(P)$ can be of the 
above `strange type'.
\smallskip

In Section~\ref{prelims} we discuss some properties of measure spaces. In 
particular, for a general measure space we recall the definitions of locally 
measurable functions and sets, and we study relations between weak 
localisability and localisability.

In Section~\ref{proj-bands} we recall properties of bands in weakly localisable 
measure spaces, and we show our main result on bands for general measure spaces.

Section~\ref{example} is devoted to an example illustrating the main theorem in 
the case of a measure space that is not weakly localisable and to an example of 
an absorption semigroup as mentioned above.
\smallskip

Throughout we assume our $L_p$-spaces to consist of real-valued
functions.

\section{Preliminaries}
\label{prelims}

Let $(\Omega,\cA,\mu)$ be a measure space, and let $1\le p< \infty$. Then for 
every function $f\in L_p(\mu)$ the set $[f\ne0]$ is $\sigma$-finite, and this 
implies that for the definition and properties of $L_p(\mu)$ only the values of 
$\mu$ on the $\sigma$-ring
\[
\cA_\sfin :=\set{A\in\cA}{A\ \sigma\text{-finite}}
\]
play a role. We also define
\[
\cA_\fin := \set{A\in\cA}{\mu(A)<\infty}.
\]
We recall that a function $f\from\Omega\to\R$ is \emph{locally 
measurable} if $\indic_A\clg f$ is measurable for all $A\in\cA_\fin$.
A set $B\sse\Omega$ is \emph{locally measurable} if $B\cap 
A\in \cA$ for all $A\in\cA_\fin$, and $B$ is a \emph{local null set} if 
additionally $\mu(B\cap A)=0$ for all $A\in\cA_\fin$. 
We say that the measure space (or the measure $\mu$) is 
\emph{weakly localisable} if for 
any consistent family $(f_A)_{A\in\cA_\fin}$ of measurable functions 
$f_A\colon A\to\R$ there exists a locally measurable function 
$f\colon\Omega\to\R$ such that $f\restrict_A=f_A$ \,$\mu$-a.e.\ for all 
$A\in\cA_\fin$. The property just introduced is called `localisable' in 
\cite[Sec.\,14,~M]{kel-nam-63}. It appears that, more recently, `localisable'
is associated with a stronger property; see our discussion below. 

It will be convenient to use the following notation. Analogously to the family 
of functions $(f_A)_{A\in\cA_\fin}$ used above we will need a family of sets 
with index set $\cA_\fin$, and for each $A\in\cA_\fin$ the associated 
set will be denoted by~$A'$, where $A'$ belongs to $\cA\cap A$, the trace 
$\sigma$-algebra of $\cA$ on $A$. In other words, the family 
$(A')_{A\in\cA_\fin}$ is a mapping \ $'\colon\cA_\fin\to\cA_\fin$, satisfying 
$A'\in\cA\cap A$ for all $A\in\cA_\fin$.

\begin{remark}\label{rem-pre}
It is a standard fact that $\mu$ is weakly localisable if 
and only if for each consistent family $(A')_{A\in\cA_\fin}$ of sets 
$A'\in\cA\cap A$ there exists a locally measurable set $\Omega_0\sse\Omega$ 
such that $A'=\Omega_0\cap A$ \,$\mu$-a.e.\ for all $A\in\cA_\fin$. 
(For $B,C\in\cA$ we use the notation `$B=C$ \,$\mu$-a.e.' if 
$\indic_B=\indic_C$ \,$\mu$-a.e.)

For the necessity one considers the consistent family 
$(\indic_{A'})_{A\in\cA_\fin}$. If $f\from\Omega\to\R$ is the corresponding 
locally measurable function, then the set $\Omega_0:=[f=1]$ satisfies 
$A'=\Omega_0\cap A$ \,$\mu$-a.e.\ for all $A\in\cA_\fin$.

For the less trivial implication, the sufficiency, we refer to 
\cite[213N]{fre-10}, \cite[Exercise~4.59]{sal-16}. Even though the context in 
these references differs from ours, the proofs can be adapted to our situation.
\end{remark}

For the information of the reader we are going to put the concept of 
weak localisability in relation to other closely related notions of measure 
theory. We emphasise that the results presented in the remainder of 
this section will not be used in the following Sections~\ref{proj-bands} 
and~\ref{example}.

We recall that the measure space $(\Omega,\cA,\mu)$ is 
\emph{semi-finite} if for all 
$B\in\cA$ with $\mu(B)=\infty$ there exists $A\in\cA\cap B$ such 
that $0<\mu(A)<\infty$. This property is equivalent to the requirement that 
every local null set belonging to $\cA$ is already a null set. (This holds 
because $\mu(A)\in\{0,\infty\}$ for all local null sets $A\in\cA$.)

From \cite[Def.~211G]{fre-10} we adopt the terminology
that $(\Omega,\cA,\mu)$ is \emph{localisable} if $\mu$ is 
semi-finite, and for each system $\cE\sse\cA$ there exists 
$A=\esssup\cE\in\cA$ (that is, $A$ is the smallest set, up to $\mu$-null 
sets, containing all $B\in\cE$, again up to $\mu$-null sets). We also 
refer to \cite{sal-16}; the concept of localisability was first introduced by 
Segal \cite{seg-51} (in a context where `measure space' is defined slightly
differently).

We recall that $\sigma$-finite measure spaces are (weakly) localisable.
More information on localisable measure spaces can be found in \cite{fre-10} 
and \cite{sal-16}.

The following proposition shows, in particular, that localisability implies 
weak localisability.

\begin{proposition}\label{lem-intro-1}
A measure space $(\Omega,\cA,\mu)$ is localisable if and only if it is 
semi-finite, and for all consistent families $(A')_{A\in\cA_\fin}$ (as in 
Remark~\ref{rem-pre}) there exists $\Omega_0\in\cA$ such that 
$A'=\Omega_0\cap A$ \,$\mu$-a.e.\ for all $A\in\cA_\fin$.
\end{proposition}

\begin{proof}
If the measure $\mu$ is localisable, then by definition it is semi-finite. 
For a consistent family $(A')_{A\in\cA_\fin}$ the set
$\Omega_0:=\esssup\sset{A'}{A\in\cA_\fin}$ is as required. This proves the
necessity. 

For the sufficiency let $\cE\sse\cA$. For all $A\in\cA_\fin$ there exists 
$A':=\esssup(\cE\cap A)$. (Here we use that $\sup\sset{\indic_{B\cap 
A}}{B\in\cE}$ exists in the order 
complete space $L_1(\mu)$ and is the indicator function of a set $A'\in\cA\cap 
A$.)
Then $(A')_{A\in\cA_\fin}$ is a consistent family, 
and by hypothesis there exists $\Omega_0\in\cA$ such that $A'=\Omega_0\cap A$ 
for 
all $A\in\cA_\fin$. 
The semi-finiteness of $\mu$ then implies
that $\Omega_0=\esssup\cE$.
\end{proof}

\begin{remark}\label{rem-prop-pre-1}
Clearly, in the `localisation condition' formulated in 
Proposition~\ref{lem-intro-1} one does not need the whole index set 
$\cA_\fin$, but it is 
sufficient to use an index set $\cA'\sse\cA_\fin$ such that for every set 
$A\in\cA_\fin$ there exists $B\in\cA'$ such that $B=A$ \,$\mu$-a.e. (The same 
comment applies to the 
definition of weak localisability and to Remark~\ref{rem-pre}.)
\end{remark}

Let us define the $\sigma$-algebra
\[
\cA_\loc := \set{B\sse\Omega}{B\text{ locally measurable}},
\]
and define
\[
\mu_\loc(B) := \sup\set{\mu(B\cap A)}{A\in\cA_\fin}\qquad (B\in\cA_\loc).
\]
Then $\mu_\loc$ is a measure, 
$\mu_\loc\restrict_{\cA_\sfin}=\mu\restrict_{\cA_\sfin}$, and $\mu_\loc(B)=0$ 
for all local null sets.
Clearly, the measure $\mu_\loc$ is semi-finite, and $\mu$ is semi-finite if and 
only if $\mu_\loc\restrict_\cA=\mu$.

We note that for any $A\in\cA_{\loc,\fin}$ there exist a set $A'\in\cA_\fin$ 
and a local null set $A''$ such that $A=A'\cup A''$.
Indeed, $\mu_\loc(A)$ can be obtained as $\lim_{n\to\infty}\mu(A_n)$, with an 
increasing sequence $(A_n)_{n\in\N}$ in $\cA_\fin$ of subsets of $A$; therefore 
$A':=\bigcup_{n\in\N}A_n\in\cA_\fin$, and $A'':=A\setminus A'$ is a local null 
set. As a consequence one also concludes that every set $A\in\cA_{\loc,\sfin}$ 
can be written as $A=A'\cup A''$, with $A'\in\cA_\sfin$ and a local null set 
$A''$.

We further observe that for all $p\in[1,\infty)$ one obtains  
$L_p(\mu)=L_p(\mu_\loc)$, with the following interpretation. It is clear that 
every function $f\in L_p(\mu)$ is also measurable with respect to the 
$\sigma$-Algebra $\cA_\loc$ and belongs to $L_p(\mu_\loc)$. On the other hand, 
if $f\in L_p(\mu_\loc)$, then $[f\ne0]\in\cA_{\loc,\sfin}$, and by the 
previous paragraph one has $[f\ne0]=A'\cup A''$, with $A'\in\cA_\sfin$ and a 
local null set $A''$. This implies that $\indic_{A'}\clg f=f$ in 
$L_p(\mu_\loc)$ and that $\indic_{A'}\clg f\in L_p(\mu)$. One then concludes 
that the 
mapping $L_p(\mu_\loc) \ni f\mapsto \indic_{A'}\clg f\in L_p(\mu)$ is an 
isomorphism.

\begin{proposition}\label{prop-intro}
The measure space $(\Omega,\cA,\mu)$ is weakly localisable if 
and only if $(\Omega,\cA_\loc,\mu_\loc)$ is localisable.
\end{proposition}

\begin{proof}
Recall that $\mu\restrict_{\cA_\fin}=\mu_\loc\restrict_{\cA_\fin}$ and that for 
every $A\in\cA_{\loc,\fin}$ there exists $A'\in\cA_\fin$ such that $A'=A$ 
\,$\mu_\loc$-a.e. This implies that, in order to verify the `localisation 
condition' of Proposition~\ref{lem-intro-1} for $\mu_\loc$, it is sufficient to 
use $\mu_\loc$-consistent families $(A')_{A\in\cA_\fin}$; 
cf.~Remark~\ref{rem-prop-pre-1}. In fact, such a family is 
$\mu_\loc$-consistent if and only if it is $\mu$-consistent.

If $\Omega_0\in\cA_\loc$, $(A')_{A\in\cA_\fin}$ is a consistent family, and 
$A\in\cA_\fin$, then $A'=\Omega_0\cap A$ \,$\mu$-a.e.\ if and only if 
$A'=\Omega_0\cap A$ \,$\mu_\loc$-a.e. 
This shows that $\mu$ is weakly localisable if and only if $\mu_\loc$ satisfies 
the `localisation condition' of Proposition~\ref{lem-intro-1}. As $\mu_\loc$ is 
semi-finite the latter is equivalent to the localisability of $\mu_\loc$.
\end{proof}

The measure space $(\Omega,\cA,\mu)$ is \emph{locally determined} if it is 
semi-finite and every locally measurable set is measurable.
One easily sees that $(\Omega,\cA,\mu)$ is locally determined if and 
only if $\cA_\loc=\cA$ and $\mu_\loc=\mu$.
This implies that the following statement is immediate from 
Proposition~\ref{prop-intro}.

\begin{corollary}\label{cor-lem-intro-1}
If the measure space $(\Omega,\cA,\mu)$ is locally determined and weakly 
localisable, then it is localisable.
\end{corollary}

\section{Bands in $L_p(\mu)$}
\label{proj-bands}

Let $(\Omega,\cA,\mu)$ be a measure space, and let $1\le p<\infty$. A 
\emph{band} $M$ in 
a vector lattice $E$ is a 
lattice ideal that additionally is stable 
under forming suprema of subsets, i.e., if $A\sse M$ is a set for which $\sup 
A$ 
exists in $E$, then $\sup A\in M$. 
We recall from \cite[Chap.\,II, \S2]{sch-74} 
that every band $M$ in $L_p(\mu)$ is automatically a \emph{projection band}, 
i.e.,
\[
L_p(\mu) = M\oplus M^\perp
\]
is the topological direct sum of $M$ and its disjoint complement
\[
M^\perp := \set{f\in L_p(\mu)}{|f|\wedge|g|=0\ (g\in M)},
\]
and then one has $M=M^{\perp\perp}$. The 
projection $P\in\cL(L_p(\mu))$ onto $M$ satisfies $0\le P\le I$, and 
conversely, 
every projection $P$ satisfying $0\le P\le I$ is a band projection 
(\cite[Thm.~1.44]{ali-bur-85}).

In the following we recall that in the case of `nice' measure spaces, 
band projections are multiplication operators by indicator functions of locally 
measurable sets. If $P$ is a band 
projection and there exists a locally measurable function $h\colon\Omega\to\R$ 
such that $Pf=hf$ for all $f\in L_p(\mu)$, then it is easy to see that 
$\Omega\setminus\bigl([h=\nobreak0]\cup[h=\nobreak1]\bigr)$ is a local null 
set. This implies
that $P$ is the multiplication operator by the indicator function of the set 
$\Omega_0:=[h=1]$, and the band $M:=\ran(P)$ is given by
\begin{equation}\label{nice-band}
M=\set{f\in L_p(\mu)}{f=0\ \,\mu\text{-a.e. on }\Omega\setminus\Omega_0}. 
\end{equation}

It is a consequence of \cite[Theorem~7]{zaa-75} that in the case of 
$\sigma$-finite measures, a band is always of the kind described in the 
previous 
paragraph. This fact is part of the assertion in the following proposition.

\begin{proposition}\label{prop-main} 
The measure $\mu$ is weakly localisable if and only 
if for every band $M$ in $L_p(\mu)$ there exists a locally measurable set 
$\Omega_0$ such that the band projection $P$ onto $M$ is given by 
$Pf=\indic_{\Omega_0}\clg f$ ($f\in L_p(\mu)$). 
\end{proposition}

\begin{proof}
In order to show the necessity we first treat the case that $\mu$ is finite. 
Then from $P1 + (I-P)1=1$ and 
$P1\wedge(I-P)1 = 0$ it is immediate that $h:=P1 = 
\indic_{\Omega_0}$ for a measurable set $\Omega_0\sse\Omega$. If $f\in 
L_p(\mu)$, 
$0\le f\le 1$, then $0\le Pf\le P1$ and $0\le(I-P)f\le(I-P)1$, and 
using also $f=Pf+(I-P)f$ one obtains $Pf=\indic_{\Omega_0}\clg f=hf$. 
By denseness of $\lin\sset{f\in L_p(\mu)}{0\le f\le 1}$ in $L_p(\mu)$ and 
continuity, the equality 
$Pf=hf$ carries over to all $f\in L_p(\mu)$.

Now assume that $\mu$ is weakly localisable. For $A\in\cA_\fin$ let $\mu_A$ 
denote the restriction of $\mu$ to 
$\cA\cap A$,
and consider $L_p(A,\mu_A)$ as as a subspace of $L_p(\mu)$.
Then it is immediate that the restriction $P_A$ of $P$ to
$L_p(A,\mu_A)$ is the band projection 
onto $M\cap L_p(A,\mu_A)$; hence -- by the first part of the proof 
-- there exists 
$A'\in\cA\cap A$ such that $P_Af=\indic_{A'}\clg f$ for all $f\in 
L_p(A,\mu_A)$.
It is not difficult to see that the family 
$(A')_{A\in\cA_\fin}$ is consistent. The weak localisability of $\mu$ -- 
together with Remark~\ref{rem-pre} -- implies that there exists a
locally measurable set 
$\Omega_0$ such that $Pf=\indic_{\Omega_0}\clg f$ for all $f\in 
L_p(\mu)$ 
with $\mu([f\ne0])<\infty$. By denseness of $\set{f\in L_p(\mu)}{\mu([f\ne 
0])<\infty}$ in $L_p(\mu)$ and continuity,
the equality $Pf=\indic_{\Omega_0}\clg f$ carries over to all $f\in L_p(\mu)$.

For the proof of the sufficiency let $(A')_{A\in\cA_\fin}$ be a consistent 
family of sets $A'\in\cA\cap A$. Then
\[
M:=\set{f\in L_p(\mu)}{f=0\ \,\mu\text{-a.e.\ on }A\setminus A'\ (A\in\cA_\fin)}
\]
is a band in $L_p(\mu)$. By hypothesis, there exists a locally measurable set 
$\Omega_0\sse\Omega$ such that the band projection~$P$ onto $M$ is given by
$Pf=\indic_{\Omega_0}\clg f$ for all $f\in L_p(\mu)$. 
In particular it follows that
$\indic_{A'}=P\indic_A=\indic_{\Omega_0}\clg\indic_A$, i.e.,
$A'=\Omega_0\cap A$ \,$\mu$-a.e., for all $A\in\cA_\fin$.
\end{proof}

As a consequence of Proposition~\ref{prop-main} one concludes that
for measure 
spaces that are not weakly localisable, there will always exist a band which
is 
not of the form described in \eqref{nice-band}. This will be illustrated by an 
example in 
Section~\ref{example}. 
In order to motivate our replacement for the
description \eqref{nice-band} we mention that this description can also be 
formulated by stating 
that there exists a locally measurable set $\Omega_0\sse\Omega$ such that for 
the measures $\mu':=\indic_{\Omega_0}\mu$ and
$\mu'':=\indic_{\Omega\setminus\Omega_0}\mu$ one has $M=L_p(\mu')$ and
$M^\perp = L_p(\mu'')$~-- in a sense made precise in 
Remark~\ref{remmain-1}(c). 

If $\mu', \mu''$ are measures on $\cA$ such that $\mu'\le\mu$, $\mu''\le\mu$, 
then $\mu', \mu''$ will be called \emph{locally disjoint}, denoted by 
$\mu'\perp_\loc\mu''$, if for all $A\in\cA_\fin$ there exist $A',A''\in\cA$, 
$A=A'\cup A''$, $A'\cap A''=\varnothing$, such that 
$\mu'(A'') = \mu''(A') = 0$.

The following theorem is the main result of this paper.

\begin{theorem}\label{mainthm}
\rma Let $\mu=\mu'+\mu''$ be a locally disjoint decomposition of $\mu$. 
Then
\begin{equation}\label{equ1}
M:=\set{f\in L_p(\mu)}{f=0\ \,\mu''\text{-a.e.}}
\end{equation}
is a band in $L_p(\mu)$, with
\begin{equation}\label{equ2}
M^\perp = \set{f\in L_p(\mu)}{f=0\ \,\mu'\text{-a.e.}}.
\end{equation}

\rmb Conversely, for each band $M$ in $L_p(\mu)$ there exists a 
locally disjoint decomposition $\mu=\mu'+\mu''$ of $\mu$ such that
$M$ and $M^\perp$ are given by \eqref{equ1} and \eqref{equ2}.
\end{theorem}

\begin{remark}\label{remmain-2}
We note that the definition of $M$ in Theorem~\ref{mainthm}(a) is 
meaningful: if $f,g\colon\Omega\to\R$ are measurable functions, $f=0$ 
\,$\mu''$-a.e.\ and $f=g$ \,$\mu$-a.e., then $\mu''([f\ne 
g])\le \mu([f\ne g])=0$, hence $g= f =0$ \,$\mu''$-a.e.
\end{remark}

\begin{proof}[P\,r\,o\,o\,f \ o\,f \ T\,h\,e\,o\,r\,e\,m\,~\ref{mainthm}]
(a) Denote the right hand side of \eqref{equ2} by $N$. 
We show that $L_p(\mu) = M \oplus N$ is a direct sum and that the 
projection $P$ onto $M$ satisfies $0\le P\le I$; this implies the assertions.

If $f\in L_p(\mu)$ satisfies $f=0$ \,$\mu''$-a.e.\ as well as $f=0$ 
\,$\mu'$-a.e., 
then $\mu([f\ne 0])=\mu'([f\ne 0])+\mu''([f\ne\nobreak0])=0$, i.e., $f=0$ 
\,$\mu$-a.e. 
This shows $M\cap N=\{0\}$.
Now let $f\in L_p(\mu)$. Then $A:=[f\ne 0]$ is $\sigma$-finite, and from 
$\mu'\lperp\mu''$ we infer that $A$ is the
disjoint union of two sets $A',A''\in\cA$ with 
$\mu'(A'')=\mu''(A')=0$. This shows 
that $f=\indic_{A'}\clg f + \indic_{A''}\clg f\in M+N$. In particular one 
obtains $Pf=\indic_{A'}\clg f$, and this implies $0\le P\le I$.

(b) Let $P$ denote the band projection onto $M$.

First let $A\in\cA_\sfin$. The restriction of  $P$ to $L_p(A,\mu_A)$ is a band 
projection, and from Proposition~\ref{prop-main} we know that 
$A$ is the disjoint union of two sets $A',A''\in\cA$ such that 
$Pf=\indic_{A'}\clg f$ ($f\in L_p(A,\mu_A)$). We define
\[
\mu'(A):=\mu(A'),\quad \mu''(A):=\mu(A'').
\]
For $A\in\cA\setminus\cA_\sfin$ we define $\mu'(A)=\mu''(A):=\infty$. From 
these definitions it is now clear that $\mu=\mu'+\mu''$, and that 
$\mu'\lperp\mu''$ (provided we know that $\mu'$ and $\mu''$ are measures).

As a preparation for the proof of the $\sigma$-additivity of $\mu', \mu''$ we 
show: if $A\sse B$ belong to $\cA_\sfin$, and $A',A'', B',B''$ are as defined 
above, then 
$A' = B'\cap A$ and $A'' = B'' \cap A$ \,$\mu$-a.e.
Indeed, let $f\in L_p(\mu)$ satisfy $[f\ne0]\sse A$. Then the previous 
definitions 
imply that $\indic_{A'}\clg f=Pf=\indic_{B'}\clg f=\indic_{B'\cap A}\clg f$ 
\,$\mu$-a.e. As this 
holds for all $f\in L_p(\mu)$ with $[f\ne0]\sse A$ one concludes that 
$A'=B'\cap A$ \,$\mu$-a.e. This implies also the 
analogous property for $A''$ and $B''$.

Now we show that $\mu',\mu''$ are $\sigma$-additive.
Recall that $\mu'=\mu''=\infty$ on
$\cA\setminus\cA_\sfin$, and note that sets in
$\cA\setminus\cA_\sfin$ cannot be obtained as a countable union of sets in
$\cA_\sfin$. Thus it is sufficient
to show the $\sigma$-additivity on $\cA_\sfin$. Let $(A_n)_{n\in\N}$ be a
disjoint sequence in $\cA_\sfin$, and
let $A:=\bigcup_{n\in\N}A_n$. Let $A',A'',A_n',A_n''$ ($n\in\N$) be the 
corresponding  sets defined above. From the previous paragraph we know that 
$A_n'=A'\cap A_n$ and $A_n''=A''\cap A_n$ \,$\mu$-a.e. Therefore 
the $\sigma$-additivity of $\mu$ shows that
\[
\sum_{n=1}^\infty\mu'(A_n)=\sum_{n=1}^\infty\mu(A_n')=\sum_{n=1}
^\infty\mu(A'\cap A_n) =\mu(A')=\mu'(A),
\]
and similarly for $\mu''$. This completes the proof that $\mu',\mu''$ are 
(locally disjoint) measures on $\cA$.

From the definition of $\mu',\mu''$ it is clear that $Pf=0$ \,$\mu''$-a.e.\ and 
$(I-P)f=0$ \,$\mu'$-a.e.\ for all $f\in L_p(\mu)$. This shows that $\mu',\mu''$ 
are as asserted.
\end{proof}

\begin{remark}\label{remmain-1}
(a) The description of bands given in Theorem~\ref{mainthm} implies the 
description in Proposition~\ref{prop-main}. Indeed, if $\mu',\mu''$ as in 
Theorem~\ref{mainthm} are locally disjoint and the measure $\mu$ is weakly 
localisable, then it follows that there exists a locally measurable set 
$\Omega_0$ such that $\mu'=\indic_{\Omega_0}\mu$ and 
$\mu''=\indic_{\Omega\setminus\Omega_0}\mu$.

(b) In the construction of $\mu'$ and $\mu''$ in the proof of part (b) of 
Theorem~\ref{mainthm} we obtained $\mu',\mu''$ with
\[
\cA_{\sfin,\mu'}=\cA_{\sfin,\mu''}=\cA_{\sfin},
\]
in particular such that $\mu'(A)=\mu''(A)=\infty$ for all $A\in\cA\setminus
\cA_\sfin$. This might not be the case for the decomposition supposed in part 
(a) of Theorem~\ref{mainthm}.

(c) We want to rephrase the assertion of Theorem~\ref{mainthm}(b) in the form 
that every band $M$ of 
$L_p(\mu)$ is of the form $M=L_p(\mu')$, with $\mu'$ as 
in Theorem~\ref{mainthm}(b). 

In order to interpret this statement we embed $L_p(\mu')$ into 
$L_p(\mu)$ as follows. For all $f\in L_p(\mu')$ there exists an element 
$\tilde f$ in the $\mu'$-equivalence class of $f$ with $\tilde f=0$ 
\,$\mu''$-a.e. (Recall that $A:=[f\ne0]\in\cA_\sfin$, $\mu'(A'')=0$, hence 
$\tilde f := \indic_{A'}\clg f = f$ as elements of $L_p(\mu')$.) The embedding 
of 
$L_p(\mu')$ into $L_p(\mu)$ is then given by $f\mapsto \tilde f$.

(d) Theorem~\ref{mainthm}(b) is remindful of 
\cite[Theorem~6]{tza-69}, \cite[Theorem~4.1]{ber-lac-74}, where it is 
shown that every band in $L_p(\mu)$ (and more generally every range of a 
contractive projection) is isomorphic to some $L_p(\nu)$. The point in our 
result is that $\mu'$ and $\mu''$ are such that $\mu=\mu'+\mu''$, 
$M=L_p(\mu')$, and $M^\perp=L_p(\mu'')$.
\end{remark}

\section{Examples}
\label{example}

The first issue of this section is to give an example of a measure space 
$(\Omega,\cA,\mu)$ with a band in $L_p(\mu)$ that cannot be described as 
$L_p(\Omega_0,\mu_{\Omega_0})$ for a locally measurable set 
$\Omega_0\sse\Omega$.
The example is a version of an example given by Fremlin 
\cite[Example~5(c)]{fre-78}.

In view of Proposition~\ref{prop-main} it obviously is redundant to present 
this example. However, our modification of Fremlin's example consists in the 
feature that one of the factors of the cartesian product $\Omega$ is the 
interval $(0,1)$, with Lebesgue measure, and it is this modification that
makes it possible to present the second example treated in this section, an 
interesting non-trivial absorption semigroup on $L_p(\mu)$.

Let $\mathfrak c$ denote the cardinality of the continuum. Let $\Gamma$ be a 
set of cardinality greater than $\mathfrak c$, and let
\[
\Omega := (0,1)\times \Gamma.
\]
On $(0,1)$ we use the Lebesgue measure $\lambda$.
On $\Gamma$ we use the $\sigma$-algebra consisting of the countable and the 
co-countable subsets of $\Gamma$ and the measure
\[
\nu(G):=
\begin{cases}
0   &\text{if }G\sse\Gamma\text{ is countable},\\
1   &\text{if }G\sse\Gamma\text{ is co-countable}.
\end{cases}
\]
The $\sigma$-algebra $\cA$ on $\Omega$ is defined by
\begin{align*}
\cA :=\bigl\{A\sse\Omega\setcol &A_x\text{ countable or co-countable 
}\bigl(x\in(0,1)\bigr),\\
&A^\gamma\text{ a Borel set in }(0,1)\ (\gamma\in\Gamma)\bigr\}, \\[-1.3\baselineskip]
\end{align*}
where \vspace{-0.3\baselineskip}
\begin{align*}
A_x &:= \set{\gamma\in\Gamma}{(x,\gamma)\in A}, 
\\
A^\gamma &:= \set{x\in(0,1)}{(x,\gamma)\in A}.
\end{align*}
For $A\in\cA$ we define
\[
\mu'(A) :=\sum_{x\in(0,1)}\nu(A_x),\quad 
\mu''(A):=\sum_{\gamma\in\Gamma}\lambda(A^\gamma),
\]
and we define $\mu:=\mu'+\mu''$. It is easy to check that $\mu',\mu''$ and 
$\mu$ are measures on~$\cA$. 

If $A\in\cA$, $\mu(A)<\infty$, then the set
\[
I':=\set{x\in(0,1)}{\Gamma\setminus A_x\text{ countable}}
\]
is finite. Then for the set
\[
A':=\set{(x,\gamma)\in A}{x\in I'}=A\cap(I'\times\Gamma)\in\cA\cap A
\]
one obtains
\[
\mu''(A')\le \mu''(I'\times\Gamma) 
=\sum_{\gamma\in\Gamma}\lambda(I') = 0
\]
and
\[
\mu'(A\setminus A')\le \mu'\Bigl(A\cap\bigl(\bigl((0,1)\setminus 
I'\bigr)\times\Gamma\bigr)\Bigr)=\sum_{x\in(0,1)\setminus I'}\nu(A_x)=0.
\]
This shows that $\mu'\lperp\mu''$.

\begin{proposition}\label{prop-example}
Let $1\le p<\infty$. With the measure space $(\Omega,\cA,\mu)$ and the
decomposition $\mu=\mu'+\mu''$ defined above let
\[
M:=\set{f\in L_p(\mu)}{f=0\ \,\mu''\text{-a.e.}} 
\]
be the band as in Theorem~\ref{mainthm}(a). Then the band projection $P$ onto 
$M$ is not a 
multiplication operator, i.e., there exists no locally 
measurable function $h\colon\Omega\to\R$ such that $Pf=hf$ ($f\in 
L_p(\mu)$).
\end{proposition}

\begin{proof}
Let $h\from\Omega\to\R$ be locally measurable, with $hf=f$ for 
all $f\in M$. 
Applying this equality to functions
$f=\indic_{\{x\}\times\Gamma} \in M$, where $x\in I$, one obtains
$h(x,\cdot) = 1$ \,$\nu$-a.e.\ for all $x\in(0,1)$.
This implies that the set $[h(x,\cdot)\ne1]$ is countable for
all $x\in(0,1)$; hence the set $[h\ne1]$ has cardinality less or
equal $\mathfrak c$.

Since $\Gamma$ has cardinality greater than $\mathfrak c$,
it follows that there exists $\gamma\in\Gamma$ such that
$h(\cdot,\gamma)=\indic_{(0,1)}$. This means that $hg=g$ for the function 
$g=\indic_{(0,1)\times\{\gamma\}}\in M^\perp$, and this shows that 
multiplication by $h$ cannot be the band projection onto $M$.
\end{proof}

Taking into account that $L_p(\Gamma,\nu)$ is isomorphic to $\R$, 
one sees that the space $M=L_p(\mu')$ is isomorphic to $\ell_p(0,1)$. 
The space $M^\perp=L_p(\mu'')$ is 
isomorphic to $\ell_p\bigl(\Gamma;L_p(0,1)\bigr)$.

\medskip

The second issue of this section is in connection with absorption semigroups. 
Starting with a positive $C_0$-semigroup $T=(T(t))_{t\ge0}$ on $L_p(\mu)$, for 
some measure space $(\Omega,\mu)$ and some $p\in[1,\infty)$, and with a locally 
measurable function $V\from\Omega\to[0,\infty]$ (the `absorption rate'), one 
defines the `absorption semigroup' $T_V$ by
\[
T_V(t) := \slim_{n\to\infty}\eul^{t(A-V\land n)}\qquad(t\ge0),
\]
where $A$ denotes the generator of $T$. (The limit exists by monotone 
convergence. We refer to \cite{voi-86}, \cite{voi-88}, \cite{are-bat-93}, 
\cite{man-vog-voi-16} for more information.) It was shown in 
\cite[Corollary~3.3]{are-bat-93} that then
\[
P := \slim_{t\to0\rlim}T_V(t)
\]
exists and is a band projection. 
One might hope that $P$ is a multiplication operator in this
particular context. 
However, we will present an example where the band corresponding to $P$ is as 
in Proposition~\ref{prop-example}.

We return to the special setting of Proposition~\ref{prop-example}. We define 
a $C_0$-semigroup $T$ on $L_p(\mu)$ as the direct 
sum of $C_0$-semigroups $T'$ and $T''$ on 
$M$ and $M^\perp$. On $M$ we 
define $T'(t)=I'$ (the identity operator on $M$) for all $t\ge0$. On $M^\perp$ 
we define $T''$ as the right translation semigroup on all the `fibers' of 
$M^\perp=\ell_p\bigl(\Gamma;L_p(0,1)\bigr)$,
\[
T''(t)f(x,\gamma) = f(x-t,\gamma)\qquad(t\ge0,\ (x,\gamma)\in\Omega),
\]
for $f\in M^\perp$ (with $f(x-t,\gamma):=0$ if $x-t\le0$).

Putting $T'$ and $T''$ together we obtain a positive $C_0$-semigroup 
$T=T'\oplus T''$ on $L_p(\mu)=M\oplus M^\perp$.

Let $W\colon(0,1)\to[0,\infty)$ be a Borel-measurable function with the 
property that $W\notin L_1(U)$ for any open subset $\varnothing\ne U\sse(0,1)$. 
Then $V(x,\gamma):=W(x)$ defines a locally measurable function 
$V\from\Omega\to[0,\infty)$.

\begin{proposition}\label{prop-example-2}
In the example described above, the absorption semigroup $T_V$ is given by
\[
T_V(t)f = \eul^{-tV}\clg f
\]
for $f\in M=L_p(\mu')$, and by
\[
T_V(t)f = 0\qquad(t>0)
\]
for $f\in M^\perp = L_p(\mu'')$. As a consequence,
\[
\smash{P := \slim_{t\to0\rlim}T_V(t)}
\]
is the band projection onto $M$.
\end{proposition}

\begin{proof} The expression for $T_V$ on $M=L_p(\mu')=\ell_p(0,1)$ is 
immediate.

In order to motivate $T_V(t)=0$ on $M^\perp$ we mention that on each fiber the 
semigroup describes right translation with the absorption rate $W$. This is 
expressed 
by the formula
\begin{equation}\label{eq-abs}
T_V(t)f(x,\gamma) = \exp\Bigl(-\int_{x-t}^tW(s)\di s\Bigr)f(x-t,\gamma)
\end{equation}
for $f\in M^\perp =L_p(\mu'')=\ell_p(\Gamma;L_p\bigl(0,1)\bigr)$. (For bounded 
$W$, the expression \eqref{eq-abs} is a special case of \cite[VI.2, 
(2.4)]{eng-nag-00}. Approximating general $W$ by $W\land n$, with $n\in\N$,  
one obtains \eqref{eq-abs} for general $W$.)
The circumstance that 
\[
\int_{x-t}^tW(s)\di s = \lim_{n\to\infty}\int_{x-t}^t(W(s)\land n)\di s=\infty
\]
for all the integrals in the exponential yields 
the result $T_V(t)f=0$.

The last assertion is then obvious.
\end{proof}

{
}
\bigskip

\noindent
Hendrik Vogt\\
Fachbereich Mathematik\\
Universit\"at Bremen\\
Postfach 330 440\\
28359 Bremen, Germany\\
{\tt 
hendrik.vo\rlap{\textcolor{white}{hugo@egon}}gt@uni-\rlap{\textcolor{white}{%
hannover}}bremen.de}\\[3ex]
J\"urgen Voigt\\
Technische Universit\"at Dresden\\
Fachrichtung Mathematik\\
01062 Dresden, Germany\\
{\tt 
juer\rlap{\textcolor{white}{xxxxx}}gen.vo\rlap{\textcolor{white}{yyyyyyyyyy}}%
igt@tu-dr\rlap{\textcolor{white}{%
zzzzzzzzz}}esden.de}

\end{document}